\documentclass{article}
\usepackage{amssymb}
\title{{\bf On Alternating and Symmetric Groups Which Are Quasi OD-Characterizable}}
\author{{\bf Ali Reza Moghaddamfar} \\[0.1cm]
{\em Faculty of Mathematics, K. N. Toosi
University of Technology,}\\
{\em P. O. Box $16315$-$1618$, Tehran, Iran.}\\[0.1cm]
{\em E-mails}:  {\tt moghadam@ipm.ir} \  and  \ {\tt
moghadam@kntu.ac.ir}}

\newenvironment{proof}{\noindent {\em {Proof}}.}{$\square$
\medskip}
\newtheorem{theorem}{Theorem}[section]
\newtheorem{definition}[theorem]{Definition}
\newtheorem{corollary}[theorem]{Corollary}

\newtheorem{proposition}[theorem]{Proposition}

\newtheorem{lm}[theorem]{Lemma}

\newtheorem{problem}[theorem]{Problem}
\newtheorem{coj}[theorem]{Conjecture}
\begin{document}
\newcommand{\f}{\frac}
\newcommand{\sta}{\stackrel}
\maketitle
\begin{abstract}
\noindent Let $\Gamma(G)$ be the prime graph associated with a
finite group $G$ and $D(G)$ be the degree pattern of $G$. A
finite group $G$ is said to be $k$-fold OD-characterizable if
there exist exactly $k$ non-isomorphic groups $H$ such that
$|H|=|G|$ and $D(H)=D(G)$. The purpose of this article is twofold.
First, it shows that the symmetric group $S_{27}$ is $38$-fold
OD-charaterizable.  Second, it shows that there exist many
infinite families of alternating and symmetric groups, $\{A_n\}$
and $\{S_n\}$, which are $k$-fold OD-characterizable with $k>3$.
\end{abstract}


\renewcommand{\baselinestretch}{1}
\def\thefootnote{ \ }
\footnotetext{{\em $2000$ Mathematics Subject Classification}:
20D05, 20D06, 20D08.}
\section{Introduction}
Throughout the article, all the groups under consideration are
{\em finite} and simple groups are {\em nonabelian}. For a
natural number $n$, we denote by $\pi(n)$ the set of all prime
divisors of $n$ and put $\pi(G)=\pi(|G|)$. The {\em spectrum}
$\omega(G)$ of a group $G$ is the set of orders of elements in
$G$. The set $\omega(G)$ determines the {\em prime graph}
$\Gamma(G)$ whose vertex set is $\pi(G)$, and two vertices $p$
and $q$ are {\em adjacent} if and only if $pq\in \omega(G)$. For
two vertices $p$ and $q$ we will write $(p\sim q)_G$ to indicate
that $p$ is adjacent to $q$ in $\Gamma(G)$. Denote by $s(G)$ the
number of connected components of $\Gamma(G)$ and by $\pi_i(G)$
$(i= 1, 2, \ldots, s(G))$, the set of vertices of its $i$th
connected component. If $2\in \pi(G)$ then we assume that $2\in
\pi_1(G)$. We recall that the set of vertices of connected
components of all finite simple groups are obtained in \cite{k}
and \cite{w}.

As usual, the {\em degree ${\rm deg}(p)$ of a vertex} $p\in
\pi(G)$ in $\Gamma(G)$ is the number of edges incident on $p$. We
denote the set of all vertices of the prime graph $\Gamma(G)$
which are joined to all other vertices by $\Lambda(G)$. If the
prime divisors of $|G|$ are $p_1<p_2<\cdots<p_k$, then we define
$ {\rm D}(G):=(\deg(p_1), \deg(p_2), \ldots, \deg(p_k)),$
 which is called the {\em degree pattern of $G$}.

Given a group $H$, denote by $h_{\rm OD}(H)$ the number of
isomorphism classes of groups $G$ such that $|G|=|H|$ and ${\rm
D}(G)={\rm D}(H)$. Clearly, there are only finitely many
isomorphism types of groups of order $n$, because there are just
finitely many ways that an $n\times n$ multiplication table can
be filled in. Hence $1\leqslant h_{\rm OD}(H)<\infty$ for any
group $H$. In terms of function $h_{\rm OD}$, groups $H$ are
classified as follows:

\begin{definition} Any group $H$ satisfying $h_{\rm OD}(H)=k$ is said to be $k$-fold
OD-characterizable. Usually, a $1$-fold OD-characterizable group
is simply called an OD-characterizable group, and it is called
quasi OD-characterizable if it is $k$-fold OD-characterizable for
some $k>1$.
\end{definition}

This article is a continuation of my investigations of the
OD-characterizability of alternating and symmetric groups
initiated in \cite{kogani}. We keep the notation created and the
conventions made therein.

In a series of articles (see \cite{kogani}, \cite{m4}, \cite{m3},
\cite{m0}
 and \cite{z2}), it has been shown that many of the
alternating and symmetric groups are $n$-fold OD-characterizable
for $n\in \{1, 2, 3, 8\}$. These results are summarized in the
following proposition.

\begin{proposition}\label{proposition1}
The following statements hold:
\begin{itemize}
\item[{\rm (a)}]
 The alternating groups $A_p$, $A_{p+1}$,
$A_{p+2}$ and the symmetric groups $S_p$ and $S_{p+1}$, where $p$
is a prime number, are OD-characterizable.

\item[{\rm (b)}] The alternating group $A_{10}$ is $2$-fold
OD-characterizable, while the symmetric group $S_{10}$ is
$8$-fold  OD-charaterizable.

\item[{\rm (c)}]  The alternating groups $A_{p+3}$, where $p\neq 7$ is a prime number less than $100$,
are OD-characterizable.

\item[{\rm (d)}] The symmetric groups ${S}_{p+3}$,
where $p\neq 7$ is a prime number less than $100$, are
OD-characterizable or $3$-fold OD-characterizable.
\end{itemize}
\end{proposition}

In addition, it was shown in \cite[Corollary 1.5]{kogani} that
all alternating groups $A_m$ for which $m\leqslant 100$, except
$A_{10}$, are OD-characterizable.
\begin{proposition}\label{p2}
All alternating groups $A_m$, where $m$ is a natural number less
than $100$, except $A_{10}$, are OD-characterizable.
\end{proposition}

These observations convinced us to propose the following
conjecture in \cite{kogani}:

\begin{coj}\label{coj-1} All alternating groups $A_m$, with $m\neq 10$, are
OD-characterizable.
\end{coj}

On the other hand, in recent years we have not found any simple
group $S$ with $h_{\rm OD}(S)\geqslant 3$. Therefore, we asked in
\cite{kogani} the following question:
\begin{problem}\label{prob-1}
Is there a simple group $S$ with $h_{\rm OD}(S)\geqslant 3$?
\end{problem}

Our recent investigations show that Conjecture \ref{coj-1} does
not hold in general. Recently, the authors showed in
\cite{mahmoodifar-sib} that the alternating group $A_{125}$
satisfying $h_{\rm OD}(A_{125})\geqslant 3$ (see also
\cite{mahmoodifar-arxiv}). Here, we will show that there exist
infinite families of alternating groups $A_m$ which are $k$-fold
OD-characterizable with $k\geqslant 3$. We notice that Problem
\ref{prob-1} is also answered positively through these examples.
\begin{theorem}\label{th-infinite-alt}
There are infinitely many alternating groups $A_m$ which satisfy
$h_{\rm OD}(A_m)>1$. In particular, there is no upper bound for
$h_{\rm OD}(A_m)$.
\end{theorem}

It is also worth mentioning that a similar description as
Proposition \ref{p2} is exhibited about OD-characterizability of
symmetric groups $S_m$, where $m$ is a natural number less than
$100$ (see \cite[Theorem 1.7]{kogani}). Nevertheless, in checking
the list of such groups, we found out that it contains a mistake
(in fact, Proposition 4.1 in \cite{kogani} asserts erroneously
that the symmetric group ${S}_{27}$ is $3$-fold
OD-charaterizable). Therefore, another result of the present
article can be stated as follows:
\begin{theorem}\label{th12} The symmetric group $S_{27}$ is $38$-fold
OD-characterizable.
\end{theorem}

Now, we give a revised list of symmetric groups in question.
\begin{corollary}\label{th1111} All symmetric groups $S_m$,
where $m$ is a natural number less than $100$, except $m=10, 27$,
are OD-characterizable or $3$-fold OD-characterizable. Moreover,
the symmetric group $S_{10}$ is $8$-fold OD-characterizable,
while the symmetric group $S_{27}$ is $38$-fold
OD-characterizable.
\end{corollary}

The following conjecture involving symmetric groups is posed in
\cite{kogani}:

\begin{coj}\label{coj-2} All symmetric groups $S_m$, with $m\neq 10$, are
OD-characterizable or $3$-fold OD-characterizable.
\end{coj}

It turns out that a negative answer to this conjecture is
provided by symmetric group $S_{27}$ (see also
\cite{mahmoodifar-sib, mahmoodifar-arxiv}). In addition, we will
get many other examples of symmetric groups which are $k$-fold
OD-characterizable with $k>3$.

\begin{theorem}\label{th-infinite-sym}
There are infinitely many symmetric groups $S_m$ which satisfy
$h_{\rm OD}(S_m)>3$. In particular, there is no upper bound for
$h_{\rm OD}(S_m)$.
\end{theorem}

We conclude the introduction with some further notation and
definitions. Given a natural number $m$, we denote by $l_m$ the
largest prime less than or equal to $m$ and we let
$\triangle(m)=m-l_m$. It is clear that $l_m=m$ (or equivalently
$\triangle(m)=0$) if and only if $m$ is a prime number. Note that
if $m>2$, then $l_m$ is always an odd prime, and so
$\triangle(m)$ is even iff $m$ is odd. In addition, from the
definition, it is easy to see that
$$l_m=l_{m-1}=l_{m-2}=\ldots=l_{m-\triangle(m)+1}.$$
We shall use the notation $\nu(m)$ (resp. $\nu_a(n)$) to denote
the number of types of groups (resp. abelian groups) of order $m$.
Clearly, $\nu_a(m)\leqslant \nu(m)$. We also denote the set of
partitions of $m$ by ${\rm Par}(m)$. It is known that for any
prime $p$, $\nu_a(p^m)=|{\rm Par}(m)|$. Finally, we use
$A_m$ and $S_m$ to denote the alternating and
symmetric group on $m$ letters, respectively. In the case when
$p\geqslant 5$ is a prime, we denote by $\mathcal{S}_{p}$ the set
of all simple groups with prime divisors at most $p$.
\section{Auxiliary results} In this section we give several definitions and auxiliary
results to be used later. The first of them is the following
definition of {\em subdirect products}.
\begin{definition}
Let $n\geqslant 2$. A subdirect product of the groups $G_1, G_2,
\ldots, G_n$ is a subgroup $G\leqslant G_1\times G_2\times\cdots
\times G_n$ of the direct product such that the canonical
projections $G\rightarrow G_i$ are surjective for all $i$.
\end{definition}

One way of obtaining a subdirect product of {\em two} groups is
via the fibre product construction. This is illustrated here for
two groups. Given some groups $G_1$ and $G_2$ with normal
subgroups $N_1$ and $N_2$ such that $G_1/N_1$ and $G_2/N_2$ are
isomorphic, we want to construct a group $G$ having a normal
subgroup $N$ isomorphic to $N_1\times N_2$ such that $G/N_2$ is
isomorphic to $G_1$ and $G/N_1$ is isomorphic to $G_2$. Notice
that we will usually identify $N_1\times 1$ with $N_1$ and
$1\times N_2$ with $N_2$. To carry out the construction, let
$\pi_1$ and $\pi_2$ be homomorphisms from $G_1$ and $G_2$ onto
some group $Q$. Now let
$$G:=\left\{(g_1, g_2)\in G_1\times G_2 \ | \ \pi_1(g_1)=\pi_2(g_2)\right\}.$$
It is easy to check that $G$ constitutes a subgroup of $G_1\times
G_2$, and the projection maps onto the coordinates map $G$ onto
$G_1$ and $G_2$, respectively. We call $G$ the {\em fibre product}
associated with $\pi_1$ and $\pi_2$ (Remak \cite{remak} called it
the {\em meromorphic product} of $G_1$ and $G_2$ with normal
subgroups $N_1$ and $N_2$.). Also, $N_1\times N_2$ is a normal
subgroup of $G$ and the map $\pi$ on $G$ defined by $\pi((g_1,
g_2))=\pi_1(g_1)$ maps $G$ onto $Q$ with kernel $N_1\times  N_2$,
so $G/(N_1\times N_2)$ is isomorphic to $Q$. In fact, we have
$$G_1/N_1\cong G/(N_1\times N_2)\cong G_2/N_2.$$
It is a basic observation that every subdirect product of $G_1$
and $G_2$ is a fibre product (or a meromorphic product) of these
groups.

Information on the adjacency of vertices in the prime graphs
associated with alternating and symmetric groups can be found in
\cite{vas-vdo, zav-maz}. Consider the function $S: \mathbb{N}
\rightarrow  \mathbb{N}$, defined as follows: $S(1) = 1$ and for
$m>1$ with prime factorization $m =\prod_{i}p_i^{\alpha_i}$,
$S(m)=\sum_{i} p_i^{\alpha_i}$. Then one has \cite[Lemma 4
]{zav-maz}:
\begin{lm}\label{adj-SA} Let $m$ and $n$ be ntaural numbers. Then there hold:
\begin{itemize}
\item[{\rm (1)}] $A_n$ has an element of order $m$ if and only if $S(m)\leqslant
n$ for odd $m$ and $S(m)\leqslant n-2$ for even $m$.
\item[{\rm (2)}] $S_n$ has an element of order $m$ if and only if
$S(m)\leqslant n$.
\end{itemize}
\end{lm}

The next two corollaries are the adjacency criteria for two
vertices of $\Gamma({A}_{n})$ and $\Gamma(S_{n})$, respectively.
\begin{corollary} \label{c-alternating}
Let $p, q\in \pi(A_n)\setminus \{2\}$. Then $(p\sim q)_{A_n}$ if
and only if $p+q\leqslant n$, while $(p\sim 2)_{A_n}$ if and only
if $p+2\leqslant n-2$.
\end{corollary}
\begin{corollary}\label{c-symmetric}
Let $p, q\in \pi(S_n)$. Then $(p\sim q)_{S_n}$ if and only if
$p+q\leqslant n$.
\end{corollary}

The Goldbach conjecture says that every even natural number $n$
greater than $4$ can be written as the sum of two odd
primes. In what follows, we will need a stronger conjecture:\\[0.3cm]
\noindent {\bf Strong Goldbach Conjecture.} {\em Every even
natural number $n$ greater than six can be written as the sum of
two distinct odd primes.}\\[0.3cm]
\indent We can now state the connection between the strong
Goldbach conjecture and the adjacency of vertices in the prime
graph of a symmetric group:
\begin{theorem} {\rm \cite[Theorem 3]{Bam}}\label{equiv-sgc} The following statements are equivalent:
\item[{\rm (1)}] the strong Goldbach conjecture is true;
\item[{\rm (2)}] for each even $n>6$, $\Gamma(S_{n-1})\subsetneqq \Gamma(S_n)$.
\end{theorem}
\begin{proof} It follows immediately from Lemma \ref{adj-SA} (2). \end{proof}

The coincidence criterion for pairwise nonisomorphic symmetric
groups (statement (1) of the Lemma \ref{newlemma}) is obtained
modulo strong Goldbach conjecture.

\begin{lm}\label{newlemma}  Let $m$ and $n$ be natural numbers with $2\leqslant m<n$. The prime
graphs of symmetric groups $S_m$ and $S_n$ are equal if and only
if $m=n-1$ and one of the following holds: \begin{itemize}
\item[{\rm (1)}] both $n$ and $n-2$ are non-prime odd numbers.
\item[{\rm (2)}] $n=4$ or $6$.
\end{itemize}
\end{lm}
\begin{proof}$^\ast$\footnotetext{$^\ast$The idea of proof was borrowed from \cite{zev}}
$(\Longrightarrow)$ In the case when $n\leqslant 6$, it is a
straightforward verification. In fact, the equality of the prime
graphs $\Gamma(S_n)$ and $\Gamma(S_{n-1})$, for $n\in \{4, 6\}$,
can be obviously verified using Corollary \ref{c-symmetric}.
Assume now that $n>6$. We first claim that $m=n-1$. If not, then
$m<n-1$ and hence one of the numbers $n$ and $n-1$ is even. Since
$n\geqslant 7$, it follows from strong Goldbach conjecture that
there exist two distinct odd primes $p$ and $q$ with $m\leqslant
n-2<p+q\leqslant n$. Hence $p$ is adjacent to $q$ in
$\Gamma(S_n)$, while $p$ is nonadjacent to $q$ in $\Gamma(S_m)$,
so these graphs are not equal. This contradiction shows $m=n-1$
as claimed.

By Theorem \ref{equiv-sgc}, we may assume that $n$ is an odd
number. Assume that $\Gamma(S_n)\neq \Gamma(S_{n-1})$. The sets of
vertices are distinct if and only if $n$ is a prime. If $n$ is a
composite number, then the sets of vertices are equal, and so the
sets of edges should be distinct. Hence there exist primes $p,
q\in \pi(S_n)$ with $n-1<p+q\leqslant n$. Then obviously $p+q=n$,
but $n$ is odd. The only possible case is $\{p, q\}=\{2, n-2\}$,
and so $n-2\in \pi(S_n)$ and $n-2$ is a prime.

$(\Longleftarrow)$ The conclusion follows immediately from
Corollary \ref{c-symmetric}.
\end{proof}

We need the following lemma to find some infinite families of
alternating and symmetric groups which are not OD-characterizable.
\begin{lm}\label{newlemmaprime}  Let $p$ be an odd prime number.
There are infinitely many natural numbers $n$ such that
$\triangle(p^n)>4$.
\end{lm}
\begin{proof} Take $n$ to be an even natural number. Clearly,
$p^n-4$ is always composite in this case, and so we need only
consider $p^n-2$. Let $a=p^2$ and $b=2$. We will now show that
given any positive integers $a$, $b$, we can find an infinite
number of values for $N$ such that $a^N - b$ is composite. If
$a^N-b$ is always composite, we are done. Otherwise, there exists
a positive integer $k$ such that  $a^k-b=q>a$ is prime. Then for
all positive integers $m$, we have
$$a^{k+(q-1)m}\equiv b \pmod{q},$$ and certainly $a^{k+(q-1)m} - b
> q$ for all $m>0$, so $a^{k+(q-1)m} - b$ must be composite for all
positive integers $m$.
\end{proof}

A possible generalization of Lemma \ref{newlemmaprime} is the
following statement, which seems intuitive to be true. Let $p$ be
an odd prime. Then, there are infinitely many positive integers
$n$ such that
$$p^n-2, \ \ p^n-4,  \ \ p^n-6, \ \ \ldots, \ \ p^n-(p-1),$$ are
composite. Alternatively, this problem can be formulated as
follows.
\begin{problem}
Let $p$ be an odd prime. Do there exist infinitely many positive
integers $n$ such that $\Delta(p^n)>p$.
\end{problem}

\section{The symmetric group $S_{27}$}
The aim of this section is to find the number of non-isomorphic
groups with the same order and degree pattern as the symmetric
group $S_{27}$. Indeed, we will show that there are 38 such
groups.

\noindent {\em Proof of Theorem $\ref{th12}$.}  Let $G$ be a group
satisfying the following conditions:
\begin{itemize}
\item[{\rm (1)}] \  $|G|=|S_{27}|=2^{23}\cdot 3^{13}\cdot 5^6\cdot
7^3\cdot 11^2\cdot 13^2\cdot 17 \cdot 19 \cdot 23$ and
\item[{\rm (2)}] \  ${\rm D}(G)={\rm D}(S_{27})=(8, 8, 7, 7, 5, 5, 4, 4, 2)$.
\end{itemize}

Under these conditions, we conclude immediately that
$\Gamma(G)=\Gamma(S_{27})$ (see also \cite[Lemma 2.15]{kogani}).
Letting $R$ be the solvable radical of $G$, we break the proof
into a number of separate lemmas.
\begin{lm}\label{solubleradical} $R$ is a $\{2, 3, 5, 7, 11\}$-group. In
particular, $G$ is nonsolvable.
\end{lm}
\begin{proof} First, we show that $R$ is a $23'$-group. Assume the contrary.
Let $23\in \pi(R)$ and let $x$ be an element of $R$ of order $23$.
Put $C=C_G(x)$ and $N=N_G(\langle x\rangle)$. By the structure of
$\Gamma(G)$, it follows that $C$ is a $\{2, 3, 23\}$-group. Since
$N/C$ is embedded in ${\rm Aut}(\langle x\rangle)\cong {\mathbb
Z}_{22}$, $N$ is a $\{2, 3, 11, 23\}$-group. Using Frattini
argument we get $G=RN$, and so $19\in \pi(R)$. Thus $R$ contains
a Hall $\{19, 23\}$-subgroup $L$ of order $19\cdot 23$. Clearly,
$L$ is cyclic, and so $(19\sim 23)_G$, which is a contradiction.

Next, we show that $R$ is a $q'$-group for $q\in \{13, 17, 19
\}$. Let $q\in \pi(R)$, $R_q\in {\rm Syl}_q(R)$ and $N=N_G(R_q)$.
Again, by Frattini argument $G=RN$ and since $R$ is a $23'$-group
we deduce that $23$ divides the order of $N$. Let $L$ be a
subgroup of $N$ of order $23$. Since $L$ normalizes $R_q$, $LR_q$
is a subgroup of order $23\cdot |R_q|$, which is abelian. But
then $(q\sim 23)_G$, which is a contradiction.

Finally, $R$ is a $\{2, 3, 5, 7, 11\}$-group, and since $R\neq
G$, it follows that $G$ is nonsolvable. This completes the proof
of lemma.
\end{proof}

In what follows, we put $\overline{G}=G/R$ and $S={\rm
Soc}(\overline{G})$. Clearly, $S$ is a direct product
$$S=P_1\times P_2\times \cdots \times P_m,$$ where the $P_i$ are
nonabelian simple groups, and we have $$S\leqslant
\overline{G}\leqslant {\rm Aut}(S).$$ We show now that $S$ is a
simple group, or equivalently $m=1$.
\begin{lm}\label{almostsg} $m=1$. In particular, $\overline{G}$ is an almost simple group.
\end{lm}
\begin{proof} By way of contradiction, let
$m\geqslant 2$. In this case $23$ does not divide $|S|$, for
otherwise $\deg(23)\geqslant 3$, which is  a contradiction. Hence, for
every $i$ we have $P_i\in \mathcal{S}_{19}$. Therefore $23\in
\pi(\overline{G})\subseteq \pi({\rm Aut}(S))$, and so $23$
divides the order of ${\rm Out}(S)$. But
$${\rm Out}(S)={\rm Out}(Q_1)\times{\rm Out}(Q_2)\times\cdots\times {\rm Out}(Q_r),$$
where $Q_i$ is a direct product of $n_i$ {\em isomorphic} copies
of a simple group $P_i$ such that $$S\cong Q_1\times Q_2\times
\cdots \times Q_r.$$ Therefore, for some $j$, $23$ divides the
order of the outer automorphism group of $Q_j$ of $n_j$ isomorphic
simple groups $P_j$. Since $P_j\in \mathcal{S}_{19}$, it follows
that $|{\rm Out}(P_j)|$ is not divisible by $23$ (see
\cite{atlas, zav}). Moreover, since ${\rm Out}(Q_j)={\rm
Out}(P_j)\wr S_{n_j}$, it follows that
$$|{\rm Out}(Q_j)|=|{\rm Out}(P_j)|^{n_j}\cdot (n_j)!.$$ This forces
$n_j\geqslant 23$, and so $2^{46}$ must divide the order of $G$,
which is a contradiction. Therefore $m=1$ and $S=P_1$. \end{proof}

\begin{lm}\label{final} There
are exactly $38$ possibilities for the group $G$.
\end{lm}
\begin{proof}
By Lemma \ref{almostsg}, we have  $S\leqslant
\overline{G}\leqslant {\rm Aut}(S)$, where $S$ is a nonabelian
simple group in $\mathcal{S}_{23}$, and so $\{13, 17, 19,
23\}\cap \pi({\rm Out}(S))=\emptyset$, (see \cite{atlas, zav}).
Now, it follows from Lemma \ref{solubleradical} and condition (1)
that
$$|S|=2^{a}\cdot 3^{b}\cdot 5^{c}\cdot
7^{d}\cdot 11^{e}\cdot 13^2\cdot 17 \cdot 19 \cdot 23,$$ where
$2\leqslant a\leqslant 23$, $0\leqslant b\leqslant 13$,
$0\leqslant c\leqslant 6$, $0\leqslant d\leqslant 3$ and
$0\leqslant e\leqslant 2$. Comparing this with the nonabelian
simple groups in $\mathcal{S}_{23}$, we obtain $S\cong {A}_{26}$
or ${A}_{27}$. We refer to \cite{zav} for the list of nonabelian
simple groups in $\mathcal{S}_{23}$. In the sequel, we deal with
two cases separately.

\begin{itemize}
\item[{\rm (1)}] $S\cong {A}_{27}$. In this case, we have
${A}_{27}\leqslant G/R\leqslant {S}_{27}$. Thus $G/R\cong
{A}_{27}$ or $G/R\cong {S}_{27}$. If $G/R\cong {A}_{27}$, then
$|R|=2$. Clearly, $R\leqslant Z(G)$ and $G$ is a central
extension of $R\cong {\mathbb Z}_2$. If $G$ splits over $R$, then
$G\cong {\mathbb Z}_2\times {A}_{27}$, otherwise we have $G\cong
{\mathbb Z}_2\cdot {A}_{27}$ (non-split extension). Next, we
assume $G/R\cong {S}_{27}$. In this case $R=1$ and so $G\cong
{S}_{27}$. Finally, in the case when $S\cong {A}_{27}$ there are
three possibilities for $G$.

\item[{\rm (2)}] $S\cong {A}_{26}$. In this case, we have
${A}_{26}\leqslant G/R\leqslant {S}_{26}$, and so $G/R\cong
{A}_{26}$ or $G/R\cong {S}_{26}$. If $G/R\cong {A}_{26}$, then
$|R|=54$. We claim that the only possibilities for $G$ are
$G=R\times {A}_{26}$, where $R$ is an arbitrary group of order
54, and $G=Q \times ({\mathbb Z}_2\cdot{A}_{26})$ where $Q$ is an
arbitrary group of order 27. In the first case, since there are
exactly 15 groups of order 54, there are 15 possibilities for
$G$. In the second case, since there are just 5 groups of order
27, there are 5 possibilities for $G$. To prove our claim, we
first observe that the automorphism group of $R$ has order
smaller than $|A_{26}|$. Now, let $C=C_G(R)$. Then $G/C$ is
isomorphic to a subgroup of ${\rm Aut}(R)$, and so $|G/C| <
|A_{26}|$. It follows that $C$ is not contained in $R$, so $CR>R$
and $CR$ is a normal subgroup of $G$. Since $G/R$ is simple, it
follows that $CR=G$. Let $D=C\cap R$. Since $D\leqslant R$, we
see that $C$ centralizes $D$, and since $D\leqslant C$, we see
that $R$ centralizes $D$, and thus $D\leqslant Z(G)$. Now $C/D$
is isomorphic to $A_{26}$ and $D\leqslant Z(C)$. Since the Schur
multiplier of $A_{26}$ has order 2, the only possibilities are
that $C'\cap D$ has order 1 or 2. If $|C'\cap D|=1$, then
$C'=A_{26}$ and $C=D\times A_{26}$, and in this case, $A_{26}$ is
a direct factor of $G$. In the remaining case, $C'=2\cdot
A_{26}$, and in this case, since $|D|$ is 2 times some power of
3, it follows that $2\cdot A_{26}$ is a direct factor of $C$ and
hence of $G$.

If $G/R\cong {S}_{26}$, then $|R|=27$. Actually, we want to find
(up to isomorphism) all groups $G$ having a normal subgroup $R$ of
order 27 such that $G/R$ is isomorphic to ${S}_{26}$. Let $M<G$
have index 2, where $M/R$ is ${A}_{26}$. Note that $M$ is unique.
Let $C$ be the centralizer of $R$ in $M$. Then $RC$ is normal in
$M$, and since $M/R$ is simple, there are only two possibilities:
$RC=R$ or $RC=M$. We want to show that $RC=M$. Suppose $RC=R$.
Now $M/C$ is embedded in ${\rm Aut}(R)$ and this group faithfully
permutes the 26 non-identity elements of $R$. This action must be
transitive or else $M/C$ is embedded in a direct product of the
form ${S}_a\times {S}_b$ where $a+b=26$ and both $a$ and $b$ are
positive. But $|M/C|\geqslant |{A}_{26}|$ and $$|{S}_a\times
{S}_b|=(a!)(b!)<(26!)/2.$$ This proves that the action is
transitive, and it follows that R is elementary abelian. Now
$$|{\rm Aut}(R)|= |{\rm GL}(3, 3)|<|{A}_{26}|\leqslant |M/C|.$$
This is a contradiction, so $RC=M$.

Let $Z=Z(R)=R\cap C$. Then $C/Z$ is isomorphic to $M/R$, which is
$A_{26}$. The Schur multiplier of ${A}_{26}$ has order 2, and it
follows that $C'\cap Z$ is trivial, and thus $C'$ is a normal
subgroup of $M$ isomorphic to ${A}_{26}$. It follows that $M$ is
the direct product of $R$ and $C'$. Let be write $A=C'$, so $A$
is isomorphic to ${A}_{26}$. Note that $A$ is characteristic in
$M$, so $A$ is normal in $G$. Now $G/R\cong {S}_{26}$ and $G/A$
has order 54, so $G$ is a subdirect product of $G_1={S}_{26}$ and
$G_2=G/A$. (Note that, every group of order 54 has a normal
subgroup of index 2.) Actually, since each of these groups has a
unique homomorphism onto the group of order $2$, it follows that
$G$ can be constructed as follows: let $\pi_1$ and $\pi_2$ be the
homomorphisms from $G_1$ and $G_2$ onto the group ${\mathbb
Z}_2$. Now, we consider the fibre product associated with $\pi_1$
and $\pi_2$, that is $$G=\{(a_1, a_2) \ | \
\pi_1(a_1)=\pi_2(a_2)\}.$$ Clearly, ${A}_{26}\times P$, where $P$
is a group of order 27, is a normal subgroup of $G$ and the map
$\pi$ on $G$ defined by
$$\pi((a_1, a_2))=\pi_1(a_1),$$ maps $G$ onto ${\mathbb Z}_2$ with
kernel ${A}_{26}\times P$, so $G/({A}_{26}\times P)$ is
isomorphic to ${\mathbb Z}_2$. This gives 15 groups, including
all direct products of ${S}_{26}$ with groups of order 27.
\end{itemize} This completes the proof of lemma and Theorem
\ref{th12}.
\end{proof}
\section{Non OD-characterizable alternating groups}
We start this section with a result of M. A. Zvezdina
\cite[Theorem]{zev} which is concerning simple groups whose prime
graphs coincide with the prime graph of an alternating simple
group. More precisely, she proved:

\begin{lm} \label{alt-n} {\rm \cite[Theorem]{zev}}
Let $G$ be an alternating group $A_n$, $n\geqslant 5$, and let $S$
be a finite simple group. Then the prime graphs of $G$ and $S$
coincide if and only if one of the following holds:
\begin{itemize}
\item[{\rm (a)}] $G=A_5$, $S=A_6$;

\item[{\rm (b)}] $G=A_6$, $S=A_5$;

\item[{\rm (c)}] $G=A_7$, $S\in \{L_2(49),U_4(3)\}$;

\item[{\rm (d)}] $G=A_9$, $S\in \{J_2,
S_6(2), O^+_8 (2)\}$;

\item[{\rm (e)}]  $G=A_n$, $S=A_{n-1}$, $n$ is odd, and the numbers $n$ and $n-4$ are
composite.
\end{itemize}
\end{lm}

Although the groups in the statement (e) of this lemma have the
same prime graph and so the same degree pattern, but they have
different orders. In fact, we have
$$|A_n|= |A_{n-1}|\times n. $$
Now, if we can choose the number $n$ such that $\pi(n)$ is
contained in the set of vertices of the prime graph $\Gamma(A_n)$
which are joined to all other vertices, then the groups $A_n$ and
$ A_{n-1}\times H$, where $H$ is an arbitrary group of order $n$,
have the same order and degree pattern. This will enable us to
give a positive answer to problem \ref{prob-1}.

Let $G$ be a finite group satisfying $|G|=|A_{m}|=m!/2$ and ${\rm
D}(G)={\rm D}(A_m)$. By \cite[Lemma 2.15]{kogani}, the prime
graph $\Gamma(G)$ coincides with $\Gamma({A}_m)$. Simply,
$\Gamma(G)$ is a graph with vertex set $\pi(G)=\{2, 3, 5, \ldots,
l_m\}$ in which two distinct vertices $r$ and $s$ are joined by
an edge iff $r+s\leqslant m$ if $r$ and $s$ are odd primes and
$r+2\leqslant m-2$ if $s=2$.

If $\triangle(m)\leqslant 2$ and $p=l_m$, then we will deal with
the alternating groups $A_p$, $A_{p+1}$ and $A_{p+2}$, which are
OD-charaterizable (see \cite[Theorem 1.5]{m-z-AC}). Therefore, we
may consider the alternating groups $A_m$ for which
$\triangle(m)\geqslant 3$.

Referring to the already mentioned fact (Proposition \ref{p2})
that all alternating groups $A_m$, $10\neq m\leqslant 100$, are
OD-characterizable, we restrict our attention to the alternating
groups $A_m$ where $m>100$. In what follows, we prove the
following result, which shows that there is an infinite family
$\{A_m\}$ of alternating groups such that $h_{\rm
OD}(A_m)\geqslant 3$.

\begin{proposition}\label{prop1} Let $m$ be an odd number satisfying
$\triangle(m)>4$ and $\pi(m)\subseteq \pi(\triangle(m)!)$. Then
$h_{\rm OD}(A_m)\geqslant 2$. In particular, if
$\triangle(3^n)>4$ (resp. $\triangle(5^n)>4$), then $h_{\rm
OD}(A_{3^n})\geqslant 16$ (resp. $h_{\rm OD}(A_{5^n})\geqslant
4$).
\end{proposition}
\begin{proof} First of all, it follows by Lemma 2.17 in \cite{kogani} that $\pi(\triangle(m)!)\subseteq
\Lambda(A_m)$, this would mean that every vertex in
$\pi(\triangle(m)!)$ and so in $\pi(m)$ is adjacent to all other
vertices in $\pi(A_m)$. Furthermore, since $m$ is an odd number
with $\triangle(m)>4$, the prime graphs $\Gamma(A_m)$ and
$\Gamma(A_{m-1})$ coincide by Lemma \ref{alt-n}. Now, if $H$ is
an arbitrary group of order $m$, then the groups $A_m$ and
$A_{m-1}\times H$ have the same order and degree pattern, and
hence
$$h_{\rm OD}(A_m)\geqslant 1+\nu(m)\geqslant 2.$$

In the case when $\triangle(3^n)>4$, it is routine to check that
$n\geqslant 7$, and so
$$h_{\rm OD}(A_{3^n})\geqslant 1+\nu(3^n)\geqslant 1+\nu_a(3^n)=1+|{\rm Par}(n)|\geqslant 1+|{\rm Par}(7)|=16.$$
Similarly, if $\triangle(5^n)>4$, then $n\geqslant 3$, and we
obtain
$$h_{\rm OD}(A_{5^n})\geqslant 1+\nu(5^n)\geqslant 1+\nu_a(5^n)=1+|{\rm Par}(n)|\geqslant 1+|{\rm Par}(3)|=4.$$
The proof is complete.
\end{proof}

{\em Proof of Theorem $\ref{th-infinite-alt}$}.  The result
follows immediately from Lemma \ref{newlemmaprime} and
Proposition \ref{prop1}. Note that, the proof of Proposition
$\ref{prop1}$ shows that there is no upper bound to $h_{\rm
OD}(A_m)$. $\Box$

By Proposition \ref{prop1}, we can find many examples of
alternating groups $A_m$ with $h_{\rm OD}(A_m)\geqslant 3$. We
point out here some of them.

\begin{itemize}
\item[{\rm (4.a)}]  {\em Some alternating groups $A_{m}$, $m\leqslant 1000$, with $h_{\rm OD}(A_m)\geqslant 3$.}

In this case, we obtain the following simple groups amongst
$A_{m}$ with $h_{\rm OD}(A_m)\geqslant 3$ (see Table 1):
$$A_{125}, \ \ A_{147}, \ \ A_{189}, \ \ A_{539}, \ \ A_{625}, \ \ A_{875},$$
and for each of these groups, we have (see Table 1):

\subitem{\rm (1)}  \   $h_{\rm OD}(A_{125})\geqslant
1+\nu(125)=6,$ (see also \cite{mahmoodifar-sib}),

\subitem{\rm (2)}  \  $h_{\rm OD}(A_{147})\geqslant 1+\nu(147)=7,$

\subitem{\rm (3)}  \  $h_{\rm OD}(A_{189})\geqslant
1+\nu(189)=14,$

\subitem{\rm (4)}  \  $h_{\rm OD}(A_{539})\geqslant 1+\nu(539)=3,$

\subitem{\rm (5)} \  $h_{\rm OD}(A_{625})\geqslant 1+\nu(625)=16,$

\subitem{\rm (6)} \  $h_{\rm OD}(A_{875})\geqslant 1+\nu(875)=6.$
\begin{center}
{\bf Table 1.}\\[0.3cm]
\begin{tabular}{llccll}
\hline $m$  &  $m-4$ &  $l_m$&  $\triangle(m)$ &
$\pi(\triangle(m)!)$ & $\nu(m)$
\\
\hline \\[-0.2cm]
 $125=5^3$ & $121=11^2$ &  $113$ &  $12$ & $2, 3, 5, 7, 11$ & $5$\\[0.2cm]
$147=3\cdot 7^2$ &  $143=11\cdot 13$ &  $139$ &  $8$ & $2, 3, 5, 7$ & $6$ \\[0.2cm]
$189=3^3\cdot 7$ & $185=5\cdot 37$ & $181$ & $8$ &  $2, 3, 5, 7$ & $13$\\[0.2cm]
$539=7^2\cdot 11$ & $535=5\cdot 107$ & $523$ & $16$ & $2, 3, 5, 7, 11, 13$ & $2$\\[0.2cm]
$625=5^4$ & $621=3^3\cdot 23$ & $619$ & $6$ & $2, 3, 5$ & $15$ \\[0.2cm]
$875=5^3\cdot 7$ & $871=13\cdot 67$ & $863$ &  $12$ & $2, 3, 5, 7, 11$ & $5$ \\
\hline
\end{tabular}
\end{center}

We shall use Lemma \ref{newlemmaprime} to show that there are
infinite families of alternating groups $A_n$, with $h_{\rm
OD}(A_n)\geqslant 3$.

\item[{\rm (4.b)}]  {\em An infinite family of alternating groups $A_{3^n}$, with $h_{\rm OD}(A_{3^n})\geqslant 136$.}

The existence of infinite number of values $n$ for which
$\triangle(3^n)>4$ is immediate from Lemma \ref{newlemmaprime}.
Indeed, if we take $n$ to be a natural number such that $n\equiv
14 \pmod{144}$, then it follows directly that $\{7, 17\}\subseteq
\pi(3^n-2)$ and similarly $\{5, 19\}\subseteq \pi(3^n-4)$, so the
numbers $3^n-2$ and $3^n-4$ are composite. This shows that
$\triangle(3^n)\geqslant 8$, for all $n\equiv 14 \pmod{144}$
(note that, this provides an alternate proof of Lemma
\ref{newlemmaprime} for $p=3$). Reasoning as in the proof of
preceding Proposition \ref{prop1}, we have
$$h_{\rm OD}(A_{3^n})\geqslant 1+\nu(3^n)\geqslant 1+\nu_a(3^n)=1+|{\rm Par}(n)|\geqslant 1+|{\rm Par}(14)|=136,$$
where ${\rm Par}(n)$ denotes the set of all partitions of $n$. In
particular, we have $h_{\rm OD}(A_{3^{14}})\geqslant 136$.

\item[{\rm (4.c)}]  {\em An infinite family of alternating groups $A_{5^n}$, with $h_{\rm
OD}(A_{5^n})\geqslant 4$.}

By Lemma \ref{newlemmaprime} again, there exist infinitely many
values of $n$ for which $\triangle(5^n)>4$. Now for every such
$n$, we have $h_{\rm OD}(A_{5^n})\geqslant 4$ by Proposition
\ref{prop1} (see also \cite{mahmoodifar-arxiv}).
\end{itemize}
\section{On the symmetric groups $S_m$ with $h_{\rm OD}(S_m)\geqslant 4$}
In this section we are looking for finite non-isomorphic groups
having the same order and degree pattern as a symmetric group.
Suppose that $G$ is a finite group satisfying $|G|=|S_{m}|=m!$
and ${\rm D}(G)={\rm D}(S_m)$, for some natural number $m$. First
of all, we conclude from \cite[Lemma 2.15]{kogani} that the prime
graph $\Gamma(G)$ coincides with $\Gamma(S_m)$. Actually,
$\Gamma(G)$ is a graph with vertex set $\pi(G)=\{2, 3, 5, \ldots,
l_m\}$ in which two distinct vertices $r$ and $s$ are joined by
an edge iff $r+s\leqslant m$. In the case when
$\triangle(m)\leqslant 1$, we deal with the symmetric groups
$S_p$ and $S_{p+1}$, which are OD-charaterizable by \cite[Theorem
1.5]{m-z-AC}. We now consider the symmetric groups $S_m$ for which
$\triangle(m)\geqslant 2$, that is
$$S_{p+2}, \ \ S_{p+3}, \ \ S_{p+4}, \ \ \ldots, \ \  S_{p+\triangle(m)}, $$
where $p=l_m$. On the other hand, in view of \cite[Theorem
1.7]{kogani} and Theorem \ref{th12}, it follows that all
symmetric groups $S_m$, where $10, 27\neq m\leqslant 100$, are
$3$-fold OD-characterizable, and for this reason we now restrict
our attention to the the symmetric groups $S_m$ where $m>100$.
Here, we first prove the following general result:

\begin{proposition}\label{prop2} Let $m$ be an odd number satisfying $\triangle(m)>4$
and $\pi(m)\subseteq \pi(\triangle(m)!)$. Then $h_{\rm
OD}(S_m)\geqslant 4$.
\end{proposition}
\begin{proof} It follows by Lemma 2.17 in \cite{kogani} that $\pi(\triangle(m)!)\subseteq
\Lambda(S_m)$, this would mean that every vertex in
$\pi(\triangle(m)!)$ and so in $\pi(m)$ is adjacent to all other
vertices in $\pi(S_m)$. Moreover, since $m$ is an odd number with
$\triangle(m)>4$, the prime graphs $\Gamma(S_m)$,
$\Gamma(S_{m-1})$, $\Gamma(A_m)$ and $\Gamma(A_{m-1})$ coincide by
Corollary \ref{c-symmetric} and  Lemma \ref{alt-n} (see also
\cite[Table 7]{kogani}). Now, if $H$ and $K$ are two arbitrary
groups of order $m$ and $2m$, respectively, then the groups $S_m$,
$\mathbb{Z}_2\times A_m$,  $\mathbb{Z}_2\cdot A_m$,
$S_{m-1}\times H$,  $(\mathbb{Z}_2\times A_{m-1})\times H$,
$(\mathbb{Z}_2\cdot A_{m-1})\times H$ and $A_{m-1}\times K$, have
the same order and degree pattern, and hence $h_{\rm
OD}(S_m)\geqslant 4$. The proof is complete.
\end{proof}

{\em Proof of Theorem $\ref{th-infinite-sym}$}.  The result
follows immediately from Lemma \ref{newlemmaprime} and Proposition
\ref{prop2}. Note that, the proof of Proposition $\ref{prop2}$
shows that there is no upper bound to $h_{\rm OD}(S_m)$. $\Box$

Considering Proposition \ref{prop2}, we can now find many examples
of symmetric groups $S_m$ satisfying  $h_{\rm OD}(S_m)\geqslant
4$. We point out here some of them.

\begin{itemize}
\item[{\rm (5.a)}]  {\em Some symmetric groups $S_{m}$, $m\leqslant 1000$, with $h_{\rm OD}(S_m)\geqslant 4$.}

As before in (4.a), we can obtain the following symmetric groups
amongst $S_{m}$, which are $k$-fold OD-characterizable with
$k\geqslant 4$ (see Table 1):
$$S_{125}, \ \ S_{147}, \ \ S_{189}, \ \ S_{539}, \ \ S_{625}, \ \ S_{875}.$$
The case $S_{125}$ had already been studied in
\cite{mahmoodifar-sib}.

\item[{\rm (5.b)}]  {\em There is an infinite family of symmetric groups
$S_{p^n}$, with $p\in \{3, 5\}$, such that $h_{\rm OD}(S_{p^n})\geqslant 4$.}

Actually, reasoning as before in (4.b) and (4,c), there are an
infinite number of values $n$ for which $\triangle(p^n)>4$, and
the result is now immediate from Proposition \ref{prop2} (see
also \cite{mahmoodifar-arxiv}).
\end{itemize}

We conclude this paper with some comments for future works:

\begin{itemize}
\item[{$(1)$}] Specify the exact value of $h_{\rm OD}(A_n)$ (resp.
$h_{\rm OD}(S_n)$) for alternating (resp. symmetric)) groups of
degree $n$ for $n\in \{125, \ 147, \ 189, \ 539, \ 625, \ 875\}$.

\item[{$(2)$}] It seems that it should be possible to prove that the value
of $h_{\rm OD}(A_n)$ (resp. $h_{\rm OD}(S_n)$) is bounded by some
function of $n$.
\end{itemize}

\begin{center}
 {\bf Acknowledgments}
\end{center}
The author would like to express his sincere thanks to Prof. I.
M. Isaacs for the very helpful comments and suggestions to
improve the contents and presentation of the manuscript. The
author also would like to thank Professor Andrei V. Zavarnitsine
and Professor Robert Styer for their valuable comments and
discussions. Finally, he also wishes to thank Dr. A. Mahmoudifar
for pointing out references \cite{mahmoodifar-sib,
mahmoodifar-arxiv} and for carefully reading the manuscript.

\end{document}